%% file: main.tex
\documentclass{amsart}
\usepackage{miyamath}
\usepackage{tikz}
\usetikzlibrary{matrix,arrows,snakes}
\usepackage{main}

\title[Log. Poincar\'e Lemma of Higher Level mod. $p$]{Logarithmic Exact Crystalline Poincar\'e Lemma of Higher Level modulo $p$.}
\author{Kazuaki Miyatani}
\address{%
    Department of Mathematics, Graduate School of Science, Hiroshima University\\
1-3-1 Kagamiyama, Higashi-Hiroshima 739-8526, Japan,\\miyatani@hiroshima-u.ac.jp}
\date{\today}
\begin{document}

\maketitle

\begin{abstract}
    Le Stum and Quir\'os proved the formal Poincar\'e lemma in crystalline cohomology
    of higher level using the jet complex,
    and applied it to give a de Rham interpretation of this cohomology.
    In this article, we prove the logarithmic version of the formal Poincar\'e lemma
    modulo $p$.
    Provided that each term of the log.\ jet complex is locally free,
    it gives the logarithmic version of the de Rham interpretation of
    the crystalline cohomology of higher level.
\end{abstract}

\input{intro.tex}

\input{part1.tex}

\input{part3.tex}

\section*{Acknowledgements}

This article is based on the latter half of the master thesis of the author.
The author would like to express his greatest gratitude to his advisor Atsushi Shiho for
sincere guidance, introducing him to the field of crystalline cohomology
and giving him a lot of comments on the master thesis.

The author is also grateful to Bernard Le Stum for the conversation during his stay in Japan, and to Pierre Berthelot for his sincere and detailed answer to my questions on the crystalline site of higher level.


\input{main.bbl}
\end{document}

%% file: intro.tex
\section*{Introduction}

The crystalline cohomology of level $m$ for each non-negative integer $m$
was introduced by Pierre Berthelot \cite{Berthelot:DMAI}
and generalizes the classical crystalline cohomology, the case $m=0$ being the classical one.
The larger $m$ gets, the more general base schemes and coefficient sheaves the cohomology theory allows.
This cohomology is based on the generalization of fundamental crystalline notions
(for example, PD structure) to ``level-$m$ versions'' (for example, PD structure of level $m$),
which makes the construction of the cohomology parallel to the classical crystalline cohomology.

Unlike the construction, we can not directly translate
the proof of cohomological properties of classical crystalline cohomology
to the level-$m$ version.
One of the most distinctive example is the crystalline Poincar\'e lemma, that is,
the de Rham interpretation of the crystalline cohomology of level $m$;
in fact, the usual de Rham complex gives the cohomology only after tensorisation of $\bQ$.
A precise de Rham interpretation for the crystalline cohomology of level $m$
is given by the jet complex of order $p^m$ (a ``crystalline version'' of the complex by Lieberman \cite{Lieberman});
this fact is called ``exact Poincar\'e lemma'', and proved by Le Stum and Quir\'os \cite{LeStum-Quiros:EPLCCHL} following the idea of Berthelot \cite{Berthelot:LI}.
Since the jet complex is unbounded and complicated, we still need another kind of local Poincar\'e lemma
given by the author to prove some basic properties of the crystalline cohomology of level $m$
such as base change theorem and finiteness \cite{Miyatani:FCCHL}.

As the classical crystalline cohomology has a generalization to logarithmic schemes,
it is natural to ask for the logarithmic version of crystalline cohomology of level $m$.
We expect that this generalization has interactions with the theory of
logarithmic $\sD^{(m)}$-module \cite{Montagnon}.

The theme of this article is the exact Poincar\'e lemma for the log.\ crystalline cohomology of level $m$,
which should be the starting point of investigating this cohomology.
The main theorem is this Poincar\'e lemma modulo $p$, that is, the following theorem.

\begin{theorem}[Corollary \ref{thm:exactpoincare}]
    Let $(S,\fa,\fb,\gamma)$ be a fine log.\ $m$-PD scheme on which $p$ is nilpotent,
    and let $X$ be a fine log.\ scheme which is log.\ smooth over $S$,
    and assume that the underlying scheme $\underline{X}$ is flat over $\underline{S}$.
    Moreover, \emph{assume that $p\sO_{\underline{X}}=0$}.
    Denote by $\dot\Omega_{X/S}^{(m),\bullet}$ the log.\ jet complex of $X$ {\rm (}Definition \ref{def:logjet}{\rm )},
    and let $E$ be a flat log.\ $m$-crystal in $\sO_{X/S}^{(m)}$-modules.
    Then, there exists an isomorphism in the derived category
    \[
    \bR{u_{X/S}^{(m)}}_*(E)\to E_X\otimes\dot\Omega_{X/S}^{(m),\bullet}.
    \]
    \label{conj}
\end{theorem}

By a standard argument, the proof reduces to showing the exactness of the complex
$0\to\sO_{\underline{X}}\to L\dot{\Omega}_{X/S}^{(m),0}\to L\dot{\Omega}_{X/S}^{(m),1}\to\dots$,
where $L\dot{\Omega}_{X/S}^{(m),\bullet}$ denotes the linearization of log.\ jet complex of $X$
(Theorem \ref{thm:trueformalLSQ}).
We prove this exactness by explicitly constructing a local homotopy of this complex;
this is where we need a complicated calculation and is the central point of our proof.
A more detailed strategy of the proof is given in Section 6.

We put some comments about the assumption that $p\sO_{\underline{X}}=0$ in Theorem \ref{conj}.
This assumption is just a technical one and should be got rid of in the future.
In fact, we conjecture that each term of the log.\ jet complex is locally free;
if this is true, then we may get rid of the assumption because
$p$ is nilpotent on $\underline{X}$.
An obstruction of proving the local freeness is the complexity of the log.\ jet complex
(the local freeness is still open even in the non-logarithmic situation \cite[0.2]{Miyatani:FCCHL}).

To conclude this introduction, we explain the structure of this article.
In the first three sections, we introduce basic notions concerning the theory of logarithmic crystalline cohomology
of level $m$ and collect some basic facts.
Section 4 is devoted to some logarithmic differential calculus.
In Section 5, we define the log.\ jet complex, which plays a central role on the log.\ exact Poincar\'e lemma,
and give an explicit desciption of this complex.
Then, in Section 6, we state the formal log.\ Poincar\'e lemma modulo $p$ and
explain the strategy of the proof.
The proof of it occupies next two sections and it is the most difficult part in this article.
In the last section, Section 9, we apply this result to prove Theorem \ref{conj}.

\section*{Conventions and Notations}

Throughout this article, we fix a prime number $p$ and a natural number $m$
(natural number means, in this article, non-negative integer).

Let $k$, $k'$ and $k''$ be natural numbers such that $k=k'+k''$. Then, three numbers
\[
\binom{k}{k'}:=\frac{k!}{k'!\,k''!},\hspace{1.5em}\mbinom{k}{k'}:=\frac{q!}{q'!\,q''!}\hspace{1em}\mathrm{and}\hspace{1em} \qbinom{k}{k'}:=\binom{k}{k'}\mbinom{k}{k'}^{-1}
\]
are often used in this article.
Here, $q$ \resp{$q'$, $q''$} denotes the integer part of $k/p^m$ \resp{$k'/p^m$, $k''/p^m$}.

Next, we introduce notation on multi-indices.
The element $(0,\dots,0,1,0,\dots,0)$ in $\bN^n$, where $1$ sits in the $i$-th component, is denoted by $\one_i$.
When $I$ is an element of $\bN^n$, its $k$-th component is denoted by $i_k$ for each $k=1,\dots,n$.
Similarly, we write $J=(j_1,\dots,j_n)$, and for a subscript $i$ we write $J_i=(j_{i,1},\dots,j_{i,n})$.
If $I\in\bN^n$, then $\hat{I}$ denotes the element $(i_1,\dots,i_{n-1})$ in $\bN^{n-1}$,
and this is identified with the element $(i_1,\dots,i_{n-1},0)$ in $\bN^n$.
Moreover, $|I|$ denotes the number $\sum_{k=1}^{n}i_k$.
If $J\leq I$, that is, if $j_k\leq i_k$ for all $k=1,\dots,n$, then the usual conventions
\[
\binom{I}{J}:=\prod_{k=1}^n\binom{i_k}{j_k},\hspace{1.5em}\mbinom{I}{J}:=\prod_{k=1}^n\mbinom{i_k}{j_k}
\hspace{1em}\mathrm{and}\hspace{1em}\qbinom{I}{J}:=\prod_{k=1}^n\qbinom{i_k}{j_k}
\]
are used.

Finally, let us fix some notation on log.\ schemes.
Let $X$ be a log.\ scheme. Then, the underlying scheme of $X$ is denoted by $\underline{X}$,
and its log.\ structure is denoted by $M_X$.
The structure sheaf $\sO_{\underline{X}}$ of $\underline{X}$ is simply denoted by $\sO_X$,
and the \'etale site $\underline{X}_{\et}$ of $X$ by $X_{\et}$.
If $S$ is a fine log.\ scheme and $X$ is a fine log.\ scheme over $S$, we agree that $X^{(n)}_{/S}$ denotes
the fiber product (in the category of the fine log.\ schemes) of $n$ copies of $X$ over $S$.

We assume that $p$ is nilpotent on all schemes appearing in this article.

%% file: part1.tex
\section{Logarithmic crystalline site of level $m$.}
\label{ss:site}

In this article, we freely use the fundamental notions of $m$-PD structures and of $m$-PD schemes;
the readers may refer to Berthelot's fundamental article \cite{Berthelot:DMAI}
or an article by the author \cite{Miyatani:FCCHL}.

Let $(S,\fa,\fb,\gamma)$ be a fine log.\ $m$-PD scheme, that is,
a datum which consists of a fine log.\ scheme $S$, a quasi-coherent ideal
$\fa$ of $\sO_S$ and a quasi-coherent $m$-PD structure $(\fb, \gamma)$ on $\fa$.

As in the classical case \cite[(5.6)]{Kato:LSFI}, we may construct the
logarithmic version of $m$-PD envelope.

\begin{proposition-definition}
Let $\mathscr{C}'_1$ denote the category of the data $(i,J,\delta)$ consisting of
an exact closed immersion $i\colon X\hookrightarrow Y$ of fine log.\ $S$-schemes
and an $m$-PD structure $(J,\delta)$ on the ideal of $\underline{X}\hookrightarrow \underline{Y}$.
Let $\mathscr{C}'_2$ denote the category of the closed immersions $i\colon X\hookrightarrow Y$ of log.\ $S$-schemes
such that $M_X$ is fine and $M_Y$ is coherent.
Then, the forgetful functor $\mathscr{C}'_1\to\mathscr{C}'_2$ has a right adjoint functor.
When $i\colon X\hookrightarrow Y$ is an object of $\mathscr{C}'_2$, its image under this functor is
called the {\rm log.\ $m$-PD envelope} of $i$ and
is denoted by $(X\hookrightarrow D^{(m)}_X(Y), \sI, { }^{[\,]})$.
\end{proposition-definition}

We fix throughout this subsection a fine log. $S$-scheme $X$,
and assume that the $m$-PD structure $(\fb,\gamma)$ extends to $\sO_X$.

\begin{definition}

    (i) Let $\underline{U}$ be an \'etale scheme over $\underline{X}$, and let $U$ denote the log.\ scheme $(\underline{U}, M_X|_{\underline{U}})$.
    A {\em log.\ $m$-PD thickening} $(U,T,J,\delta)$ of $U$ over $(S,\fa,\fb,\gamma)$ is a 
    datum which consists of a fine log.\ $S$-scheme $T$, an exact closed $S$-immersion $U\hookrightarrow T$
    and an $m$-PD structure $(J,\delta)$ on the ideal of $\underline{U}\hookrightarrow \underline{T}$ compatible with $(\fb,\gamma)$.

    (ii) A log.\ $m$-PD thickening $(U, T, J, \delta)$ is said to be {\em fundamental} if
    there exists a fine log.\ smooth $S$-scheme $Y$ and a closed $S$-immersion $i\colon U\hookrightarrow Y$
    such that $(U, T, J, \delta)$ is isomorphic to the log.\ $m$-PD envelope of $i$.
\end{definition}

When $(U,T,J,\delta)$ is a fundamental log.\ $m$-PD thickening, the structure of $\sO_T$ is described by the following lemma.

\begin{lemma}
Let $i\colon U\hookrightarrow Y$ be a closed immersion of fine log.\ smooth schemes over $S$
such that the underlying scheme $\underline{U}$ is flat over $\underline{S}$.
Then for an integer $n$, the structure sheaf $\sO_{D^{(m)}_U(Y)}$ of the log.\ $m$-PD envelope of $i$ is, \'etale locally on $Y$,
isomorphic to the $m$-PD polynomial algebra $\sO_U\{t_1,\dots,t_n\}_{(m)}$ over $\sO_U$,
which is introduced in Berthelot's article {\rm \cite[1.5.1]{Berthelot:DMAI}}.
\label{lem:rest-D}
\end{lemma}

\begin{proof}
We may assume that $i$ is an exact closed immersion and that there exists a chart $(P\to M_S, Q\to M_Y, Q\to P)$ of
$i$ which satisfies the condition in \cite[(3.5)]{Kato:LSFI}.
Then, this chart also induces one of $U\to S$.
If $\sI$ denotes the ideal of $\underline{U}\hookrightarrow \underline{Y}$,
the second fundamental exact sequence \cite[2.1.3]{Nakkajima-Shiho:WFLCCFOSV}
\[
0\to\sI/\sI^2\to\Omega^1_{Y/S}\to\Omega^1_{U/S}\to 0
\]
shows that we may take sections $x_1,\dots,x_{n}\in\sI$ and $m_1,\dots,m_{n'}\in M_U$ so that
$dx_1$, \dots, $dx_{n}$, $\dlog m_1$, \dots, $\dlog m_{n'}$ form a basis of $\Omega^1_{Y/S}$.

Therefore, as in the arguments in \cite[(3.13)]{Kato:LSFI}, we have the cartesian diagram
\[
\begin{tikzpicture}
    \matrix (m) [matrix of math nodes,
                 row sep=3em, column sep=3.5em, text height=1.5ex, text depth=0.25ex]
        { U & T \\
          S\times_{\bZ[Q]}\bZ[P] & (S\times_{\bZ[Q]}\bZ[P])\times_S\bA_S^n, \\
        };
    \path[->,font=\scriptsize]
        (m-1-1) edge node[auto] {$i$} (m-1-2)
                edge (m-2-1)
        (m-2-1) edge node[auto] {$i'$} (m-2-2)
        (m-1-2) edge (m-2-2);
\end{tikzpicture}
\]
where the vertical morphisms are strict and \'etale, and the lower horizontal map $i'$ is the base change of the zero section
$\underline{S}\to\bA_{\underline{S}}^{n'}$.
Now we have shown that the ideal $\sI$ is a regular ideal,
and the proof is a consequence of \cite[1.5.3]{Berthelot:DMAI}.
\end{proof}

Now, we definethe logarithmic version of $m$-crystalline site and of
restricted $m$-crystalline site \cite{Berthelot:LAM}.

\begin{definition}

(i) The {\em log.\ $m$-crystalline site} $\Crism(X/S,\fa,\fb,\gamma)$, or simply $\Crism(X/S)$, is the category of the
log.\ $m$-PD thickenings $(U,T,J,\delta)$ of an \'etale scheme $U$ of $X$ over $(S,\fa,\fb,\gamma)$,
morphisms defined in an obvious way,
and topology induced by the \'etale topology on $T$.
Its associated topos $(X/S,\fa,\fb,\gamma)\crism$, or $(X/S)\crism$, is called the {\em log.\ $m$-crystalline topos}.

(ii) The {\em restricted log.\ $m$-crystalline site} $\RCrism(X/S,\fa,\fb,\gamma)$, or $\RCrism(X/S)$, is the full subcategory
of $\Crism(X/S,\fa,\fb,\gamma)$ consisting of the fundamental log.\ $m$-PD thickenings,
which is equipped with the induced topology.
Its associated topos $(X/S,\fa,\fb,\gamma)\Rcrism$, or $(X/S)\Rcrism$, is called the {\em restricted log.\ $m$-crystalline topos}.
When $E$ is a sheaf in $\Crism(X/S)$ \resp{$\RCrism(X/S)$}, then for each log.\ $m$-PD thickening $(U,T,J,\delta)$
\resp{fundamental log.\ $m$-PD thickening $(U,T,J,\delta)$}, the sheaf on $T_{\et}$ induced by $E$ is denoted by
$E_{(U,T,J,\delta)}$ or simply by $E_T$.

(iii) The sheaf of rings
\[
(U,T,J,\delta)\mapsto \Gamma(\underline{T},\sO_T)
\]
in the topos $(X/S)\crism$ \resp{$(X/S)\Rcrism$} is called the {\em structure sheaf} of the site $\Crism(X/S)$ \resp{$\RCrism(X/S)$}, and denoted by $\sO_{X/(S,\fa,\fb,\gamma)}^{(m)}$ or by $\sO_{X/S}^{(m)}$.

\end{definition}

\begin{proposition}
Let $(S',\fa',\fb',\gamma')\to (S,\fa,\fb,\gamma)$ be an $m$-PD morphism from another fine log.\ $m$-PD scheme,
$X'$ a fine log.\ scheme over $S'$ and $f\colon X'\to X$ a morphism over $S$.
Then, there exists a morphism of topoi
\[
f\crism\colon (X'/S')\crism\to (X/S)\crism.
\]
\end{proposition}

Next, we discuss some fundamental functors.
Let $f\colon X\to S$ denote structure morphism. Then,
\[
u_{X/S}^{(m)}\colon (X/S)\crism\to X_{\et}
\]
denotes the projection of log.\ $m$-crystalline topos defined as in the classical case \cite[(6.4)]{Kato:LSFI}.
Its composite with $f_{\et}\colon X_{\et}\to S_{\et}$ is denoted by $f_{X/S}^{(m)}\colon (X/S)\crism\to S_{\et}$.
Moreover, we define the morphism of topoi $\bar{u}_{X/S}^{(m)}$ as
\[
\bar{u}_{X/S}^{(m)}:=u_{X/S}^{(m)}\circ Q_{X/S}^{(m)}\colon (X/S)\Rcrism\to X_{\et},
\]
where $Q_{X/S}^{(m)}\colon (X/S)\Rcrism\to(X/S)\crism$ denotes the obvious morphism of topoi \cite[IV 2.1.2]{Berthelot:CCSCP}.

\begin{lemma}
Let $E$ be a sheaf in $\Crism(X/S)$. Then, there exists a canonical isomorphism
\[
{u_{X/S}^{(m)}}_\ast(E)\to\left.\bar{u}^{(m)}_{X/S}\right._{\!\ast}\left({Q_{X/S}^{(m)}}^\ast(E)\right).
\]
\end{lemma}

\begin{proof}
For each fundamental thickening $(U,D_U^{(m)}(Y),J,\delta)$ in $\RCrism(X/S)$, where $Y$ is a fine log.\ smooth scheme over $S$,
we have an exact sequence
\begin{equation}
    \begin{tikzpicture}
        \matrix (m) [matrix of math nodes, row sep=3em, column sep=2.5em, text height=1.5ex, text depth=0.3ex]
        { 
        {u_{X/S}^{(m)}}_\ast(E)|_U & E_{D_U^{(m)}(Y)} & E_{D_U^{(m)}(Y^2)}. \\
        };
        \path[->]
        (m-1-1) edge (m-1-2)
        (m-1-2.5) edge (m-1-3.175)
        (m-1-2.-5) edge (m-1-3.185);
    \end{tikzpicture}
    \label{eqn:fund_thick}
\end{equation}
Indeed, the proof for the classical case \cite[IV 2.3.2]{Berthelot:CCSCP} can be generalized directly.

Now, the lemma is proved as in \cite[IV 2.3.5]{Berthelot:CCSCP}.
\end{proof}

\begin{proposition}
\label{thm:restricted}
Let $M^{\bullet}$ be a complex of $\sO_{X/S}^{(m)}$-modules which is bounded below.
Then, the canonical morphism
\[
\bR{u_{X/S}^{(m)}}_\ast(M^{\bullet})\to \bR\left.\!\!\bar{u}_{X/S}^{(m)}\right._{\!\ast}\left({Q_{X/S}^{(m)}}^\ast(M^{\bullet})\right)
\]
is an isomorphism in $D^+(X_{\et}, \sO_X)$.
\end{proposition}

\begin{proof}
The source of this morphism is quasi-isomorphic to the \v{C}ech--Alexander complex of $M^{\bullet}$,
which can be proved just as in \cite[V 1.2.5]{Berthelot:CCSCP} by using (\ref{eqn:fund_thick}).
Similarly, the target is quasi-isomorphic to that of ${Q_{X/S}^{(m)}}^*(M^{\bullet})$,
which is nothing other than the \v{C}ech--Alexander complex  of $M^{\bullet}$.
\end{proof}

\section{Crystals, differential operators  and stratifications.}

In this section, we discuss the notion of log.\ $m$-crystal, log.\ hyper $m$-PD differential operator and log.\ hyper $m$-PD stratification.
Let $(S,\fa,\fb,\gamma)$ be a fine log.\ $m$-PD scheme, and $X$ a fine log.\ scheme over $S$ such that
the $m$-PD structure $(\fb,\gamma)$ extends to $\sO_X$.

\begin{definition}
Let $E$ be an $\sO_{X/S}^{(m)}$-module in $(X/S)\crism$
\resp{$\sO_{X/S}^{(m)}$-module in $(X/S)\Rcrism$}.
Then, $E$ is called a {\em log.\ $m$-crystal} \resp{a {\em restricted log.\ $m$-crystal}}
in $\sO_{X/S}^{(m)}$-modules if for all morphism $f\colon (U,T,J,\delta)\to (U',T',J',\delta')$
of $\Crism(X/S)$ \resp{$\RCrism(X/S)$}, 
the canonical morphism
\[
f^*(E_{(U',T',J',\delta')})\to E_{(U,T,J,\delta)}
\]
is an isomorphism.
\end{definition}

In this article, $P_{X/S}^{(m)}$ denotes the log.\ $m$-PD envelope of the diagonal immersion $X\hookrightarrow X\times_SX$,
and $\P_{X/S}^{(m)}$ denotes the structure sheaf of $\underline{P_{X/S}^{(m)}}$.
We always regard $X\times_SX$ and $P_{X/S}^{(m)}$ as log.\ schemes over $X$ by the first projection.

\begin{definition}
    Let $M$ and $N$ be two $\sO_X$-modules. Then, a {\em log.\ hyper $m$-PD differential operator} from $M$ to $N$
    is an $\sO_X$-linear morphism
    \[
        \P_{X/S}^{(m)}\otimes_{\sO_X} M\to N.
    \]
\end{definition}

\begin{definition}
Let $M$ be an $\sO_X$-module. Then, a {\em log.\ hyper $m$-PD stratification} on $M$ is a $\P_{X/S}^{(m)}$-linear isomorphism
\[
    \varepsilon\colon \P_{X/S}^{(m)}\otimes_{\sO_X} M\to M\otimes_{\sO_X} \P_{X/S}^{(m)}
\]
which induces the identity map on $M$ by passing the quotient $\P_{X/S}^{(m)}\to\sO_{X}$
and satisfies the usual cocycle condition.
\end{definition}

Now, let $(\fa_0,\fb_0,\gamma_0)$ be a quasi-coherent $m$-PD sub-ideal of $\fa$,
and let $S_0\hookrightarrow S$ denote the exact closed immersion defined by $\fa_0$, and
$i\colon X_0\hookrightarrow X$ its base change by $X\to S$.
We assume that $X$ is log.\ smooth over $S$ and that the underlying scheme $\underline{X}$ is flat over $\underline{S}$.

The first important proposition in this situation is the following one,
which is proved as in non-logarithmic situation \cite[1.3.4]{Miyatani:FCCHL}.

\begin{proposition}
\label{thm:lift}
In the situation above, the functor
\[
{i\crism}_*\colon (X_0/S)\crism\to (X/S)\crism
\]
is exact. Moreover, the image of the structure sheaf $\sO_{X_0/S}^{(m)}$ under this functor is isomorphic to $\sO_{X/S}^{(m)}$,
and the image of a log.\ $m$-crystal in $\sO_{X_0/S}^{(m)}$-modules is a log.\ $m$-crystal in $\sO_{X/S}^{(m)}$-modules.
\end{proposition}

\begin{proposition}
\label{thm:crystal}
    In the situation above, the following categories are equivalent:

    {\rm (i)} the category of the log.\ $m$-crystals in $\sO_{X_0/S}^{(m)}$-modules;

    {\rm (i)'} the category of the restricted log.\ $m$-crystals in $\sO_{X_0/S}^{(m)}$-modules;

    {\rm (ii)} the category of the $\sO_X$-modules equipped with a log.\ hyper $m$-PD stratification.
\end{proposition}

\begin{proof}
First, note that $D_{X_0}(X)=X$ and that $D_{X_0}(X^{(n+1)}_{/S})=P_X(n)$ for each $n\geq 1$.
The equivalence of categories from the category of (i)' to that of (ii) is constructed in the usual manner \cite[6.6]{Berthelot-Ogus:NCC}
(in loc. cit., the source of the functor is the category of (i), but the same argument works).
Construction of the equivalence of categories from the category of (ii) to that of (i) is again classical.
\end{proof}

\begin{corollary}
    The functor ${i\crism}_\ast$ in Proposition \ref{thm:lift} induces an equivalence of categories
    from the category of the log.\ $m$-crystals in $\sO_{X_0/S}^{(m)}$-modules to
    that of the log.\ $m$-crystals in $\sO_{X/S}^{(m)}$-modules.
    The quasi-inverse of this functor is ${i\crism}^\ast$.
    \label{cor:cryst_equiv}
\end{corollary}

\begin{proof}
The argument of \cite[6.7]{Berthelot-Ogus:NCC} works.
\end{proof}

\section{Linearization.}
\label{ss:linearization}

Here, we discuss the linearization.  
Let $(S,\fa,\fb,\gamma)$ be a fine log.\ $m$-PD scheme, and $X$ a fine log.\ scheme over $S$ such that
the $m$-PD structure $(\fb,\gamma)$ extends to $\sO_X$.

First, $j_X$ signifies the localization morphism
\[
j_X\colon (X/S)\crism|_X\to (X/S)\crism,
\]
where the source denotes the localized category of $(X/S)\crism$ over the log.\ $m$-PD thickening $(X,X,0)$
with the trivial PD structure on $0$.
Then, composing with $u_{X/S}^{(m)}$, we get the morphism of topoi
\[
u_{X/S}^{(m)}|_X\colon (X/S)\crism|_X\to X_{\et}.
\]

Now, we define the linearization functor as
\[
L^{(m)}={j_X}_*\circ {u_{X/S}^{(m)}|_X}^*\colon X_{\et}\to (X/S)\crism,
\]
and its restricted version as
\begin{equation*}
    \bar{L}^{(m)}={Q_{X/S}^{(m)}}^*\circ L^{(m)}\colon X_{\et}\to (X/S)\Rcrism.
\end{equation*}
The $\sO_X$-module $\bar{L}^{(m)}(\sF)_X$ is also denoted by $\bar{L}^{(m)}_X(\sF)$.

\begin{proposition}
\label{thm:linearization}
Assume that $X$ is log.\ smooth and that the underlying scheme $\underline{X}$ is flat over $\underline{S}$.
Let $\sF$ denote an $\sO_X$-module.

{\rm (i)} $\bar{L}^{(m)}(\sF)$ is a restricted log.\ $m$-crystal, and $\bar{L}^{(m)}_X(\sF)=\P_{X/S}^{(m)}\otimes \sF$.

{\rm (ii)} We have $R{u_{X/S}^{(m)}}_\ast L^{(m)}(\sF)=R\left.\bar{u}^{(m)}_{X/S}\right._{\!\ast}\bar{L}^{(m)}(\sF)=\sF$.

{\rm (iii)} If $E$ is a restricted log.\ $m$-crystal, there exists a canonical isomorphism
\[
    E\otimes_{\sO_{X/S}^{(m)}} \bar{L}^{(m)}(\sF)\to \bar{L}^{(m)}(E_X\otimes\sF).
\]
\end{proposition}

\begin{proof}
By the construction of $L^{(m)}$, for each $(U,T,J,\delta)\in\Crism(X/S)$ we have
\[
L^{(m)}(\sF)_{(U,T,J,\delta)}={p_T}_*{p_X}^*(\sF),
\]
where $p_T$ \resp{$p_X$} denotes the projection from $D_U^{(m)}(T\times_S X)$ to $T$ \resp{to $X$}.
In particular, the latter half of (i) holds.

As for the former half, by following the classical argument \cite[IV 3.1.6]{Berthelot:CCSCP},
it is sufficient to show that the natural morphism
\begin{equation}
\label{eqn:restrictedbasechange}
D^{(m)}_U(T\times_SX)\to T\times_X P_{X/S}^{(m)}
\end{equation}
is an isomorphism for all fundamental thickenings $(U,T,J,\delta)\in\RCrism(X/S)$ such that
a retraction $T\to X$ exists.
Now, because $\sO_T$ and $\sP_{X/S}^{(m)}$ are isomorphic to some $m$-PD polynomial algebras,
so is $\sO_T\otimes\sP_{X/S}^{(m)}$.
Therefore the ideal of the closed immersion $U\hookrightarrow T\times_X P_{X/S}^{(m)}$ has an $m$-PD structure
compatible with $\gamma$.
Hence a morphism $T\times_X P_{X/S}^{(m)}\to D^{(m)}_U(T\times_SX)$ is obtained, and
it is a standard argument to show that this is an inverse morphism of (\ref{eqn:restrictedbasechange}).

In order to prove (ii), note that
\[
    R{u_{X/S}^{(m)}}_\ast L^{(m)}(\sF)=\sF
\]
by the arguments in \cite[4]{LeStum-Quiros:EPLCCHL}. Moreover, we have
\[
    R\left.{\bar{u}_{X/S}^{(m)}}\right._{\!\ast}\bar{L}^{(m)}(\sF) = R\left.{\bar{u}_{X/S}^{(m)}}\right._{\!\ast}{Q_{X/S}^{(m)}}^*L^{(m)}(\sF)
    = R{u_{X/S}^{(m)}}_*(\sF)
\]
in virtue of Proposition \ref{thm:restricted}.

For (iii), let $(U,T,J,\delta)$ be a fundamental log.\ $m$-PD thickening in $\RCrism(X/S)$.
Then because $E$ is a restricted $m$-crystal, we have the isomorphism
\[
    E_T\otimes\sO_{D^{(m)}_U(T\times X)}=E_{D^{(m)}(T\times U)}=\sO_{D^{(m)}_U(T\times X)}\otimes E_X.
\]
Tensoring $\sF$ to the left and the right sides, we see the required isomorphism.
\end{proof}

Here, we set some notations.
The $m$-PD envelope of the diagonal immersion $X\hookrightarrow X^{(r+1)}_{/S}$ is denoted by $P^{(m)}_{X/S}(r)$,
and its structure sheaf is denoted by $\sP^{(m)}_{X/S}(r)$.
We have therefore $P^{(m)}_{X/S}=P^{(m)}_{X/S}(1)$ and $\sP_{X/S}^{(m)}=\sP_{X/S}^{(m)}(1)$.
Again, we regard $X^{(r+1)}_{/S}$ and $P_{X/S}^{(m)}(r)$ as log.\ schemes over $X$ by the first projection,
which lets $\sO_X$ act on $\sO_{X^{(r+1)}_{/S}}=\sO_X\otimes_{\sO_S}\dots\otimes_{\sO_S}\sO_X$ by multiplication to the first factor
and induce on $\sP_{X/S}^{(m)}(r)$ an $\sO_X$-algebra structure.

For each natural number $r$ and $i\in\{0,\dots,r\}$, let $j_r^i\colon X^{(r+1)}_{/S}\to X$ denote
the $(i+1)$-st projection.
Let $d_r^i\colon P^{(m)}_{X/S}(r+1)\to P^{(m)}_{X/S}(r)$ ($0\leq i\leq r+1$) denote the morphism
corresponding to
\[
(j_{r+1}^0,\dots, j_{r+1}^{i-1}, j_{r+1}^{i+1},\dots, j_{r+1}^{r+1})_S\colon X^{(r+2)}_{/S}\to X^{(r+1)}_{/S},
\]
and $s_r^i\colon P^{(m)}_{X/S}(r-1)\to P^{(m)}_{X/S}(r)$ ($0\leq i\leq r$) the one corresponding to
\[
(j_{r-1}^0,\dots, j_{r-1}^i, j_{r-1}^i,\dots, j_{r-1}^r)_S\colon X^{(r)}_{/S}\to X^{(r+1)}_{/S}.
\]

Then, these data make $P_{X/S}^{(m)}(\bullet)$ a simplicial log.\ scheme over $S$, and consequently
$\sP_{X/S}^{(m)}(\bullet)$ is a DGA (differential graded algebra) over $f^{-1}(\sO_S)$;
the differential morphism $d^r\colon\sP_{X/S}^{(m)}(r)\to \sP_{X/S}^{(m)}(r+1)$ is by definition
\[
d^r=\sum_{i=0}^{r+1} (-1)^i{d_r^i}^*.
\]

Now, we proceed to a calculation of the hyper $m$-PD stratification on $\sP_{X/S}^{(m)}\otimes_{\sO_X}\sF=L_X^{(m)}(\sF)$,
which exists because of Proposition \ref{thm:linearization} (i) and Proposition \ref{thm:crystal}. 

\begin{lemma}
    \label{lem:inducingstrat}
    Assume that $X$ is log.\ smooth over $S$ and that the underlying scheme $\underline{X}$ is
    flat over $\underline{S}$.
    Then, the log.\ $m$-PD stratification on $\bar{L}_X^{(m)}(\sF)=\P_{X/S}^{(m)}\otimes_{\sO_X}\sF$ is induced by 
    \[
        \sO_{X^{(2)}_{/S}} (\sO_{X^{(2)}_{/S}} \otimes \sF)\to (\sO_{X^{(2)}_{/S}}\otimes\sF)\otimes\sO_{X^{(2)}_{/S}};\hspace{1em}
    (1\otimes 1)\otimes(f\otimes g)\otimes x\mapsto (1\otimes g)\otimes x\otimes(1\otimes f),
    \]
    where the tensor products are taken over $\sO_X$.
\end{lemma}

\begin{proposition}
    Let $M$ and $N$ be two $\sO_X$-modules and $u$ a log.\ hyper $m$-PD differential operator from $M$ to $N$.
    Then, the morphism
    \[
    \begin{tikzpicture}
        \matrix (m) [matrix of math nodes,
                     row sep=3em, column sep=3.5em, text height=1.5ex, text depth=0.25ex]
                     { \sP_{X/S}^{(m)}\otimes M & \sP_{X/S}^{(m)}\otimes\sP_{X/S}^{(m)}\otimes M &
                     \sP_{X/S}^{(m)}\otimes N \\
                    };
        \path[->,font=\scriptsize]
        (m-1-1) edge node[auto] {${d_1^1}^*\otimes \id$} (m-1-2)
        (m-1-2) edge node[auto] {$\id\otimes u$} (m-1-3);
    \end{tikzpicture}
    \]
    is compatible with the log.\ hyper $m$-PD stratifications on both sides viewed as $\P_{X/S}^{(m)}\otimes M=L^{(m)}_X(M)$ and as
    $\P_{X/S}^{(m)}\otimes N=L^{(m)}_X(N)$.
\end{proposition}

%% file: part3.tex
\section{Logarithmic differential calculus.}
\label{ss:a-differential}

Let $(S,\fa,\fb,\gamma)$ be a fine log.\ $m$-PD scheme and $X$ be a fine log.\ scheme over $S$
such that the $m$-PD structure $(\fb,\gamma)$ extends to $\sO_X$.
Then, the fine log.\ scheme $P_{X/S}^{(m)}(r)$ is defined just as in the previous section.
Notations $\sP_{X/S}^{(m)}(r)$, $d^i_r$, $s^i_r$, $d^r$ and $N\sP_{X/S}^{(m),\bullet}$ are also defined as in that section.

Now, assume that $X$ is log.\ smooth over $S$ and that we are given a family of sections $t=(t_1,\dots,t_n)$ of $M_X$ such that
$(\dlog t_i)_{1\leq i\leq r}$ forms a basis of log.\ differential sheaf $\Omega_{X/S}^1$.
We call such a family ``global coordinates of $X$ over $S$''.

In this situation, there exists a unique family $(u_i)_{1\leq i\leq r}$ of sections in $\P_{X/S}^{(m)}$ such that
$j_1^*(t_i)=j_0^*(t_i)u_i$.
We set $\eta_i:=u_i-1$ and $\eta:=(\eta_1,\dots,\eta_n)$.
For $I=(i_1,\dots,i_n)\in\bN^n$, the following notation is used:
\[
\eta^{\{I\}}:=\prod_{j=1}^n \eta_j^{\{i_j\}}.
\]

\begin{proposition}
\label{thm:logPm}
The differential map $d^1\colon\P^{(m)}_{X/S}\to\P^{(m)}_{X/S}(2)$ is described by
\[
d^1\left(\eta^{\{I\}}\right)=-\sum_{\substack{A+B+C=I\\ B,C\neq I}}\Gamma_{A,B,C}\,\eta^{\{A+B\}}\otimes\eta^{\{A+C\}},
\]
where
\[
\Gamma_{A,B,C} := \qbinom{A+B+C}{A}\qbinom{B+C}{B}\mbinom{A+B}A\mbinom{A+C}A.
\]
\end{proposition}

\begin{proof}
    Since $d^1=d_1^{0\ast}-d_1^{1\ast}+d_1^{2\ast}$ and
    \[
        d_1^{0\ast}\big(\tau^{\{I\}}\big)=\tau^{\{I\}}\otimes 1 \quad\text{and}\quad
        d_1^{2\ast}\big(\tau^{\{I\}}\big)=1\otimes\tau^{\{I\}},
    \]
    it suffices to calculate ${d_1^1}^*(\eta^{\{I\}})$.
    Now, we know \cite[p.16]{Ogus:FCGTHD} that ${d_1^1}^*(\eta_i)=\eta_i\otimes 1+\eta_i\otimes\eta_i+1\otimes\eta_i$ for each $i$.
    Therefore \cite[1.3.6 (iii)]{Berthelot:DMAI},
    \begin{eqnarray*}
        {d_1^1}^*\left(\eta^{\{I\}}\right) &=& (\eta\otimes 1+\eta\otimes\eta+1\otimes\eta)^{\{I\}}\\
        &=& \sum_{A+B+C=I}\qbinom{I}{A}\qbinom{I-A}{B}(\eta\otimes 1)^{\{A\}}(\eta\otimes\eta)^{\{B\}}(1\otimes\eta)^{\{C\}},
    \end{eqnarray*}
    and we have \cite[1.3.6 (iv)]{Berthelot:DMAI}
    \[
        (\eta\otimes 1)^{\{A\}}(\eta\otimes\eta)^{\{B\}}(1\otimes\eta)^{\{c\}}
        = \mbinom{A+B}{A}\mbinom{A+C}{A}\eta^{\{A+B\}}\otimes\eta^{\{A+C\}}.
    \]
    This completes the proof.
\end{proof}

For later use, we here present some lemmas on binomial coefficients and the number $\Gamma_{A,B,C}$.

\begin{lemma}
\label{lem:0modp}
Assume that natural numbers $q, k$ and $t$ satisfies $q\geq 1$ and $0\leq k<t\leq p^m$.
Then, the number $\qbinom{p^mq+k}{t}$ is

{\rm (i)} congruent to $0$ modulo $p$ if $t<p^m$, and

{\rm (ii)} congruent to $1$ modulo $p$ if $t=p^m$.
\end{lemma}

\begin{proof}
    (i) The equation [B1, (1.1.3.1)] shows that the $p$-adic valuation of this number is equal to
    \[
    \frac{a(t)+a(p^m+k-t)-a(k)-1}{p-1}=\frac{a(t)+a(p^m+k-t)-a(p^m+k)}{p-1},
    \]
    where $a(t)$ denotes the sum of the coefficients of the $p$-adic expansion of $t$.
    Since this is the $p$-adic valuation of $\binom{p^m+k}{t}$,
    it is strictly greater than $0$ by assumption.

    (ii) We have $\qbinom{p^mq+k}{p^m}=(1/q)\binom{p^mq+k}{p^m}$, and this equals
    \[
        \frac{1}{q}\frac{p^mq+k}{k}\frac{p^mq+k-1}{k-1}\dots\frac{p^mq+k}{k}\frac{p^mq}{q}\frac{p^mq-1}{p^m-1}\dots\frac{p^mq-p^m+k+1}{k+1}.
    \]
    For $0<r<p^m$, the number $(p^mq+r)/r$ is congruent to $1$ modulo $p$, which shows the assertion.
\end{proof}

\begin{lemma}
\label{lem:1modp}
Let $q$ and $t$ be natural numbers, and assume that $0\leq t\leq p^m$.
Then, the number $\qbinom{p^mq+t}{t}$ is congruent to $1$ modulo $p$.
\end{lemma}

\begin{proof}
    If $t=0$, we have $\qbinom{p^mq}{0}=1$ and the lemma is obvious.

    If $0<t<p^m$, we have
    \[
        \qbinom{p^mq+t}{t} = \binom{p^mq+t}{t}=\frac{p^mq+t}{t}\frac{p^mq+t-1}{t-1}\dots\frac{p^mq+1}{1}.
    \]
    Since each fraction in the right hand side is congruent to $1$ modulo $p$, so is this number itself.

    If $t=p^m$, the assertion is a consequence of Lemma \ref{lem:0modp} (ii).
\end{proof}

\begin{lemma}
\label{lem:b+p^mq}
Let $q$ and $k$ be natural numbers, and assume that $0\leq k<p^m$, that $a+b+c=p^mq+k$ and that $0\leq a+c\leq p^m$.
Then in $\bZ/p\bZ$, the number $\Gamma_{a,b,c}$ is equal to zero  except in the following three cases:

{\rm (i)} if $(a,b,c)=(p^m,p^m(q-1)+k,0)$, then $\Gamma_{a,b,c}= q$;

{\rm (ii)} if $(a,b,c)=(0,p^m(q-1)+k,p^m)$, then $\Gamma_{a,b,c}= 1$;

{\rm (iii)} if $b\geq p^mq$, then $\Gamma_{a,b,c}=\Gamma_{a,b-p^mq,c}$.
\end{lemma}

\begin{proof}
    Firstly, note that the assumption $a+c\leq p^m$ shows that
    \begin{equation}
        \Gamma_{a,b,c}=\qbinom{p^mq+k}{c}\binom{p^mq+k-c}{a}.
        \label{eq:gammaabc}
    \end{equation}
    In fact, $\mbinom{a+c}{a}=1$ shows that
    \[
        \Gamma_{a,b,c}=\binom{p^mq+k}{a}\binom{p^mq+k-a}{b}\mbinom{p^mq+k}{a}^{-1}\mbinom{p^mq+k-a}{b}^{-1}\mbinom{a+b}{a}.
    \]
    By the direct calculation, the product of the last three numbers equals $\mbinom{p^mq+k}{c}^{-1}$,
    which shows (\ref{eq:gammaabc}).
    
    Secondly, we check the equations in the exceptional three cases. 
    (i) and (ii) follows from Lemma \ref{lem:0modp} (ii).
    For (iii), we compute
    \[
    \Gamma_{a,b,c}=\frac{(p^mq+k)!}{a!b!c!}=\frac{k!}{a!(b-p^mq)!c!}\frac{(p^mq+k)!}{k!(p^mq)!}\frac{(b-p^mq)!(p^mq)!}{b!}=\Gamma_{a,b-p^mq,c}
    \]
    with the aid of Lemma \ref{lem:1modp}. This completes the proof. 
    
    Finally, assume that we are not in the exceptional cases, and let us show that $\Gamma_{a,b,c}=0$.
    It directly follows from Lemma \ref{lem:0modp} if $0\leq k<c$.
    If $a=0$, then $\Gamma_{a,b,c}=\qbinom{p^mq+k}{c}$ and it vanishes again by Lemma \ref{lem:0modp}(i)
    and by $c<p^m$ (note that the case $c=p^m$ belongs to the exceptional cases).
    If $a\neq 0$ and $c\neq 0$, then $\binom{p^mq+k-c}{a}$ is congruent to $0$ modulo $p$
    because $0\leq k-c<a$,
    whence so is $\Gamma_{a,b,c}$.
\end{proof}

\section{Definition of log.\ jet complex.}
\label{ss:logjet}

Now, let us define the log.\ jet complex, which is a logarithmic version of the jet complex defined by Le Stum
and Quir\'os \cite{LeStum-Quiros:EPLCCHL}.

Let $S$ be a fine log.\ scheme and $X$ a fine log.\ smooth scheme over $S$ such that the $m$-PD structure $(\fb,\gamma)$ extends to $\sO_X$.

\begin{definition}
    \label{def:logjet}
    Let $\sK$ be the differential ideal of $N\P_{X/S}^{(m),\bullet}$ generated by the sections $\eta^{\{I\}}$,
    where $I$ runs through the elements of $\bN^n$ such that $|I|>p^m$.
    Then, the {\em log.\ jet complex} $\dot\Omega_{X/S}^{(m),\bullet}$ is by definition the quotient complex
    \[
    N\P_{X/S}^{(m),\bullet}/\sK,
    \]
    where the differential map is induced by that of $N\P_{X/S}^{(m),\bullet}$.
    The {\em linearized log.\ jet complex} $L\dot\Omega_{X/S}^{(m),\bullet}$ is defined to be the quotient complex
    \[
    \P_{X/S}\otimes N\P_{X/S}^{(m),\bullet}/\P_{X/S}\otimes\sK.
    \]
    Here, we agree that its differential map $d^r\colon L\dot\Omega_{X/S}^{(m),r}\to L\dot\Omega_{X/S}^{(m),r-1}$ is by definition
    \[
    d^r := \sum_{i=1}^{r+1}(-1)^{i+1}{d_r^i}^*.
    \]
\end{definition}

Notice that, when $\delta\in L\dot\Omega_{X/S}^{(m),1}$ and $\delta',x\in\P_{X/S}^{(m)}$, we have
\begin{equation}
    \delta\otimes x\delta'=x\delta\otimes\delta'+d^0(x)\delta\otimes\delta'
    \label{eqn:trans}
\end{equation}
(here, $d^0$ in the right-hand side is viewed as the differential of the linearized log.\ jet complex).
In fact, right $\P_{X/S}^{(m)}$-module structure of $L\dot\Omega_{X/S}^{(m),1}$ is ${d_1^1}^*$.

Now, we give a local description of the linearized log.\ jet complex $L\dot\Omega_{X/S}^{(m),\bullet}$.
Assume that $X$ has global coordinates $t=(t_1,\dots,t_n)$ over $S$.
In order to describe the structure of this complex, we here introduce some notations.
For $i=1,\dots,n$, let $\dlog t_i$ denote the image of $1\otimes \eta_i$ by the natural surjection 
$\sP_{X/S}^{(m)}\otimes N\sP_{X/S}^{(m),1}\to L\dot\Omega_{X/S}^{(m),1}$.
For $I, J_1, \dots, J_r\in\bN^n$, we put
\[
\delta(I;J_1,\dots,J_r) := \eta^{\{I\}}(\dlog t)^{J_1}\otimes\dots\otimes(\dlog t)^{J_r},
\]
where $(\dlog t)^{J_k}$ denotes the product $\prod_{l=1}^n(\dlog t_l)^{j_{k,l}}$
(see Conventions for the notation concerning multi-indices;
in particular, the $l$-th component of $J$ is denoted by $j_l$, and that of $J_k$ by $j_{k,l}$).
Therefore, if $|J_i|>p^m$ for some $i$, then $\delta(I;J_1,\dots,J_r)=0$.

Similarly, we define
\[
\hat\delta(I;J_1,\dots,J_r) := \hat\eta^{\{I\}}(\dlog\hat{t})^{J_1}\otimes\dots\otimes(\dlog\hat{t})^{J_r}
\]
for $I, J_1, \dots, J_r\in\bN^{n-1}$ where
\[
\hat\eta^{\{I\}}:=\prod_{l=1}^{n-1}\eta_l^{\{i_l\}}\hspace{4pt}{\rm and}\hspace{4pt}(\dlog\hat{t})^{J_k}=\prod_{l=1}^{n-1}(\dlog t_l)^{j_{k,l}},
\]
and also
\[
\delta_n(i; j_1,\dots,j_r) := \eta_n^{\{i\}}(\dlog t_n)^{j_1}\otimes\dots\otimes(\dlog t_n)^{j_r}
\]
for $i, j_1, \dots, j_r\in\bN$.

For $J_1,\dots,J_r\in\bN^n$, the number $s(J_1,\dots,J_r)$ is defined as follows:
\[
s(J_1,\dots,J_r) := 
\begin{cases}
	\min\left\{s\in [1,r]~\mid~j_{s,n}\neq 0\right\} & \mathrm{if}~ j_{s,n}\neq 0 ~{\rm for}~{\rm some}~ s\in [1,r],\\
	r & \mathrm{if}~ j_{s,n}=0 ~{\rm for}~{\rm all}~ s\in [1,r].
\end{cases}
\]

Let $\fI_n^{(m)}$ denote the set of elements $I\in\bN^n$ such that $0<|I|\leq p^m$.
Then for each $r$, the $\sO_X$-module $L\dot\Omega_{X/S}^{(m),r}$ is generated by
$\delta(I; J_1,\dots,J_r)$ for $I\in\bN^n$ and $J_1,\dots,J_r\in\fI_n^{(m)}$.
Their relations are
\begin{equation}
\label{eqn:relation}
-\delta(I;J_1,\dots,J_{k-1})\otimes d^1\left(\eta^{\{J\}}\right)\otimes\delta(0;J_{k+2},\dots,J_r)=0,
\end{equation}
where $I,J\in\bN^n$, $J_1,\dots,J_{k-1},J_{k+2},\dots,J_r\in\fI_n^{(m)}$ with $p^m<|J|\leq 2p^m$:
here, $d^1$ denotes the differential of Proposition \ref{thm:logPm}.
For convenience, we give a name for this relation.

\begin{definition}
    The \emph{relation of type $(s,k)$}, with $0\leq s,k\leq r$, is the section of $L\dot\Omega_{X/S}^{(m),r}$
    of the form
    \[
    -\delta(I;J_1,\dots,J_{k-1})\otimes d^1\left(\eta^{\{J\}}\right)\otimes\delta(0;J_{k+2},\dots,J_r),
    \]
    where $I, J\in\bN^n$, where $J_1,\dots,J_{k-1},J_{k+2},\dots,J_r\in \fI_n^{(m)}$ with $p^m<|J|\leq 2p^m$
    and where $s=s(J_1,\dots,J_{k-1},J_{k+2},\dots,J_r)$.
\end{definition}

Proposition \ref{thm:logPm} shows that the relation of type $(s,k)$ equals
\[
    \sum_{\substack{A+B+C=J\\ B,C\neq J}}\Gamma_{A,B,C}\,\delta(I;J_1,\dots,J_{k-1},A+B,A+C,J_{k+2},\dots,J_r).
\]

Finally, we describe the differential map.
As in the calculation in Proposition \ref{thm:logPm}, we can show that the morphism $d^0\colon L\dot\Omega_{X/S}^{(m),0}\to L\dot\Omega_{X/S}^{(m),1}$ satisfies
\[
d^0\left(\eta^{\{I\}}\right)=\sum_{\substack{A+B+C=I\\ B\neq I}}\Gamma_{A,B,C}\,\delta(A+B; A+C),
\]
and that $d^1\colon L\dot\Omega_{X/S}^{(m),1}\to L\dot\Omega_{X/S}^{(m),2}$ satisfies
\[
d^1\left( (\dlog t)^J\right)=-\sum_{\substack{A+B+C=J\\ B,C\neq J}}\Gamma_{A,B,C}\,\delta(0; A+B, A+C).
\]
In general, differential maps $d^r$ satisfy the usual Leibniz rule, hence we may calculate them
by using these two equations.

\section{Formal Log.\ Poincar\'e Lemma Modulo $p$.}

We may now state the main theorem of this article, the log.\ crystalline Poincar\'e lemma of higher level
modulo $p$.

\begin{theorem}
Let $(S,\fa,\fb,\gamma)$ be a fine log.\ $m$-PD scheme,
and $X$ be a fine log.\ smooth scheme over $S$ such that the $m$-PD structure $(\fb,\gamma)$ extends to $\sO_X$.
Assume that $p\sO_X=0$.
Then, the log.\ jet complex $L\dot\Omega_{X/S}^{(m),\bullet}$ is a resolution of $\sO_X$.
\label{thm:trueformalLSQ}
\end{theorem}

In order to prove this theorem, the question being local, we assume that $X$ has global coordinates $t=(t_1,\dots,t_n)$.

\begin{proposition}
\label{thm:formalLSQ}
Fix an index $i\in\{1,\dots,n\}$, and let $\pi_i$ be the projector on $L\dot\Omega_{X/S}^{(m),\bullet}$ as a morphism of graded $\sO_X$-algebra defined by
\[
    \pi_i(\eta_j)  = \begin{cases} 0 & \text{ if } i=j\\ \eta_j & \text{ if } i\neq j\end{cases} \qquad \text{and}\qquad
    \pi_i(\dlog t_j)  = \begin{cases} 0 & \text{ if } i=j\\ \dlog t_j & \text{ if } i\neq j.\end{cases}
\]
Then, there exists a family of $\sO_X$-linear morphisms
\[
    \left\{ h^r_i\colon L\dot\Omega_{X/S}^{(m),r}\to L\dot\Omega_{X/S}^{(m),r-1}\right\}_{r\geq 1}
\]
which is a homotopy connecting $\id_{L\dot\Omega_{X/S}^{(m),\bullet}}$ and $\pi_i$,
that is, which satisfies the following equations:
\begin{align*}
& (h^{r+1}_i\circ d^{r}+d^{r-1}\circ h^r_i)\big(\delta(I;J_1,\dots,J_r)\big) \\
= & \begin{cases}
0 & \text{ if } i_n=j_{1,i}=\dots=j_{r,i}=0,\\
\delta(I;J_1,\dots,J_r) & \text{ otherwise}.
\end{cases}
\end{align*}
\end{proposition}

If we admit this proposition, Theorem \ref{thm:trueformalLSQ} is easily proved.
Indeed, Proposition \ref{thm:formalLSQ} shows that the identity map of $\id_{L\dot\Omega_{X/S}^{(m),\bullet}}$ is
homotopic to $\pi_1\circ\pi_2\circ\dots\circ\pi_n$, which maps all $\eta_i$'s and $(\dlog t_i)$'s to zero.

In order to prove Proposition \ref{thm:formalLSQ}, we may assume that $i=n$ after a change of indices.
The following two sections are devoted to proving Proposition \ref{thm:formalLSQ} for $i=n$;
We define morphisms $h^r_n$ in Section \ref{ss:homotopy},
and prove that they satisfy the properties of Proposition \ref{thm:formalLSQ} in Section \ref{ss:logPL}.
Since $i$ is fixed to be $n$, we simply write $h^r$ for $h^r_n$.

\section{Construction of a homotopy modulo $p$.}
\label{ss:homotopy}

Let $X$ be a fine log.\ smooth scheme over $S$ such that the $m$-PD structure $(\fb,\gamma)$ extends to $\sO_X$,
and assume that $X$ has global coordinates $t=(t_1,\dots,t_n)$ and that $p\sO_X=0$.

\begin{proposition}
\label{thm:homotopy}
There uniquely exists a family of $\sO_X$-linear morphisms
\[
    \left\{h^r\colon L\dot\Omega_{X/S}^{(m),r}\to L\dot\Omega_{X/S}^{(m),r-1}\right\}_{r\geq 1}
\]
that satisfies the following two conditions.

{\rm (i)} The morphism $h^1$ satisfies $h^1\big(\delta(I;J)\big)=0$ if $i_n$ is not divisible by $p^m$.
When $i_n=p^mq$ with a natural number $q$, it satisfies
\[
h^1\big(\delta(I; J)\big)=
\begin{cases}
	\hat\eta^{\{\hat{I}\}}\eta_n^{\{i_n+j_n\}} & \mathrm{if}~ \hat{J}=0~\mathrm{and}~0<j_n<p^m,\\
	\displaystyle\sum_{u=q}^{\sigma(q)-1}(-1)^{u-q}\frac{\,u!\,}{q!}\hat\eta^{\{\hat{I}\}}\eta_n^{\{p^m(u+1)\}} & \mathrm{if}~\hat{J}=0~\mathrm{and}~j_n=p^m,\\
    0 & \mathrm{if}~\hat{J}\neq 0~\mathrm{or}~j_n=0.
\end{cases}
\]
Here, $\sigma(q)$ denotes the least multiple of $p$ which is strictly greater than $q$.

{\rm (ii)} For $r\geq 2$, let $I\in\bN^n$ and $J_1, \dots, J_r\in\fI_n^{(m)}$, and write $s:=s(J_1,\dots,J_r)$.
Then, if $0<j_{s,n}<p^m$, or if $j_{s,n}=p^m$ and $s=r$, the morphism $h^r$ maps $\delta(I; J_1,\dots,J_r)$ to
\[
(-1)^{s-1}\delta(\hat{I}; J_1,\dots,J_{s-1})\otimes h^1\big(\delta(i_n\one_n; J_s)\big)\delta(0;J_{s+1},\dots,J_r).
\]
\end{proposition}

This section is devoted to the proof of this proposition.
Here, we temporarily introduce two quotient graded $\sO_X$-modules 
$M^{\bullet}$, $\overline{M}^{\bullet}$ of $\sP_{X/S}^{(m)}\otimes N\sP_{X/S}^{(m), \bullet}$, that fit in the sequence of canonical surjections
\[
    \sP_{X/S}^{(m)}\otimes N\sP_{X/S}^{(m),\bullet}\to M^{\bullet}\to\overline{M}^{\bullet}
    \to L\dot\Omega_{X/S}^{(m),\bullet}.
\]
Firstly, $M^{\bullet}$ is defined as
\[
    M^{\bullet} := \sP_{X/S}^{(m)}\otimes \left(N\sP_{X/S}^{(m),\bullet}/ \sK'\right),
\]
where $\sK'$ denotes the ideal of $N\P_{X/S}^{(m),\bullet}$ generated by the sections $\eta^{\{I\}}$
for all $I\in\bN^n$ such that $|I|>p^m$.
Therefore, the $r$-th part $M^r$ is the free $\sO_X$-module with basis $\delta(I; J_1,\dots,J_r)$,
where $I\in\bN^n$ and $J_1,\dots, J_r\in\fI_n^{(m)}$.
Secondly, $\overline{M}^{\bullet}$ is defined to be the quotient of $M^{\bullet}$ by its submodule generated by all relations
of type $(k,k)$, for some $k$, such that $j_n\geq p^m$;
in other words, each term $\overline{M}^r$ is the quotient of $M^r$ by the sections of the form
\[
    \delta(I; J_1,\dots,J_{k-1})\otimes d^1\big(\eta^{J}\big)\otimes\delta(0;J_{k+2},\dots,J_r),
\]
where $I, J\in\bN^{n}$, $J_1,\dots,J_{k-1},J_{k+2},\dots,J_r\in\fI_n^{(m)}$
with $j_{1,n}=\dots=j_{k-1,n}=0$ and $j_n\geq p^m$.

Consequently, the kernel of the $r$-th term of the last surjection is generated by the relations
of type $(s,k)$ such that $s\neq k$ or that $s=k$ and $j_n<p^m$.

In fact, we can explicitly describe the structure of the modules $\overline{M}^r$.

\begin{lemma}
    For each natural number $r$, the $\sO_X$-module $\overline{M}^r$ is free.
    The sections $\delta(I;J_1,\dots,J_r)$ form a basis of $\overline{M}^r$ when $I$ runs through $\bN^n$
    and $J_1,\dots,J_r$ run through $\fI_n^{(m)}$ so that $j_{s(J_1,\dots,J_r),n}<p^m$ or $s(J_1,\dots,J_r)=r$.
    
    In particular, there exists a unique morphism of $\sO_X$-modules $h^r\colon\overline{M}^r\to L\dot\Omega_{X/S}^{(m),r-1}$ which satisfies
    the same conditions as in Proposition \ref{thm:homotopy}.
    \label{lem:generatorbarM}
\end{lemma}

\begin{proof}
By definition, the kernel of the natural surjection $M^r\to\overline{M}^r$ is generated by
\begin{equation}
    \label{eqn:kernel}
    \sum_{\substack{A+B+C=J\\ B,C\neq J}}\Gamma_{A,B,C}\,\delta(I;J_1,\dots,J_{k-1},A+B,A+C,J_{k+2},\dots,J_r),
\end{equation}
where $J_1,\dots,J_{k-1}\in\fI_{n-1}^{(m)}$, $I,J\in\bN^n$ and $J_{k+2},\dots,J_r\in\fI_n^{(m)}$
such that $p^m<|J|\leq 2p^m$ and $j_n\geq p^m$.
Let $N^r$ denote the submodule of $M^r$ generated by the sections stated in the lemma,
and $\overline{N}^r$ its image in $\overline{M}^r$.
Then, the section (\ref{eqn:kernel}) is congruent modulo $N^r$ to
\begin{align*}
    &  \sum_{\substack{a+b+c=j_n\\ a+b=p^m,\\ b,c\neq j_n}}\Gamma_{a,b,c}\,\delta(I;J_1,\dots,J_{k-1},(a+b)\one_n,\hat{J}+(a+c)\one_n,J_{k+2},\dots,J_r)\\
    & = \sum_{a=0}^{2p^m-j_n}\Gamma_{a,p^m-a,j_n-p^m}\,\delta(I;J_1,\dots,J_{k-1},p^m\one_n,\hat{J}+(a+j_n-p^m)\one_n,J_{k+2},\dots,J_r).
\end{align*}
because, for each term of (\ref{eqn:kernel}), $\hat{A}+\hat{B}$ must be zero if the $n$-th component of $A+B$ equals $p^m$.
Now, for $I\in\bN^n$ and $J_1,\dots,J_r\in\fI_n^{(m)}$ such that $j_{s(J_1,\dots,J_r),n}=p^m$ and that 
$s(J_1,\dots,J_r)<r$,
we see that the section $\delta(I;J_1,\dots,J_r)$ belongs to $\overline{N}^r$ by descending induction on $j_{s(J_1,\dots,J_r)+1,n}$.
This shows that $\overline{N}^r=\overline{M}^r$.

In fact, different relations (\ref{eqn:kernel}) are used to express different $\delta(I;J_1,\dots,J_r)$'s as a linear combination
of the generators of $\overline{N}^r$.
Therefore, these generators form a basis of $\overline{N}^r$.

The latter half follows directly from the first half.
\end{proof}

\begin{lemma}
\label{lem:omega2}
Assume that $r\geq 2$, and let $h^r\colon \bar{M}^r\to L\dot\Omega_{X/S}^{(m),r-1}$ be the $\sO_X$-linear morphism
defined in Lemma \ref{lem:generatorbarM}.
Then, for all $s<r$, for all $I\in\bN^n, J_1,\dots,J_{s-1}\in\fI_n^{(m)}$ with
$j_{1,n}=\dots=j_{s-1,n}=0$, and for all $J_{s+1},\dots,J_r\in\fI_n^{(m)}$, we have the equation
\begin{align*}
& h^r\big(\delta(I; J_1,\dots,J_{s-1},p^m\one_n,J_{s+1},\dots,J_r)\big) \\
= & -h^r\big(\delta(I; J_1,\dots,J_{s-1},J_{s+1},p^m\one_n,J_{s+2},\dots,J_r)\big).
\end{align*}
\end{lemma}

\begin{proof}
The morphism $h^r$ sends a relation (\ref{eqn:relation}) of type $(k,k)$ to
\begin{align}
    \label{eqn:relbar}
	&h^r\big(\delta(I;J_1,\dots,J_{k-1},\hat{J},j_n\one_n,J_{k+2},\dots,J_r)\big)\nonumber\\
	+&\sum_{\substack{a+b+c=j_n\\ c\neq j_n}}\Gamma_{a,b,c}h^r\big(\delta(I;J_1,\dots,J_{k-1},(a+b)\one_n,\hat{J}+(a+c)\one_n,J_{k+2},\dots,J_r)\big)
\end{align}
if $j_n\geq p^m$ (if $\hat{J}=0$, we interprete the first term to be zero).
Under the notation in the statement of the lemma, we put $k=s$ and
\[
    J=J_{s+1}+p^m\one_n=\hat{J_{s+1}}+(j_{s+1,n}+p^m)\one_n,
\]
We prove that the lemma follows from the vanishing of (\ref{eqn:relbar}) for this $k$ and $J$.

First, assume that $j_{s+1,n}>0$.
Then, the first term of (\ref{eqn:relbar}) (for $k$ and $J$ above) equals zero.
In the second term, $\Gamma_{a,b,c}$ is zero unless $(a,b,c)=(p^m,j_n-p^m,0),(0,j_n-p^m,p^m)$ or $b\geq p^m$ by Lemma \ref{lem:b+p^mq}.
Since $\delta(I;J_1,\dots,(a+b)\one_n,\dots,J_r)$ is equal to zero if $a+b>p^m$, we conclude that
the summand is zero unless $(a,b,c)=(0,p^m,j_n-p^m)$ or $(0,j_n-p^m,p^m)$.
Hence, the section (\ref{eqn:relbar}) is equal to
\begin{eqnarray*}
    && h^r\big(\delta(I;J_1,\dots,J_{k-1},p^m\one_n,\hat{J}+(j_n-p^m)\one_n,J_{k+2},\dots,J_r)\big)\\
    &+& h^r\big(\delta(I;J_1,\dots,J_{k-1},(j_n-p^m)\one_n,\hat{J}+p^m\one_n,J_{k+2},\dots,J_r)\big).
\end{eqnarray*}
If $\hat{J}=0$, the vanishing of this section shows the lemma.
Next, if $\hat{J}\neq 0$, the second term vanishes, hence the vanishing of this section
this lemma.

Second, assume that $j_{s+1,n}=0$. In this case, $\hat J$ is necessarily non-zero.
In the each summand of the second term of (\ref{eqn:relbar}), the coefficient $\Gamma_{a,b,c}$ vanishes unless $(a,b,c)=(0,p^m,0)$, $(p^m,0,0)$
in virtue of Lemma \ref{lem:b+p^mq} (note that $a+b>0$).
If $(a,b,c)=(p^m,0,0)$, the section $\delta(I;J_1,\dots,\hat{J}+(a+c)\one_n,\dots,J_r)$ equals zero,
therefore we conclude that the vanishing of (\ref{eqn:relbar}) shows the lemma.

This completes the proof of the lemma.
\end{proof}

\begin{corollary}
\label{cor:initialh}
Let $I\in\bN^n$ and $J_1,\dots,J_r\in\fI_n^{(m)}$, and let $s'$ be a natural number such that $s'\leq s:=s(J_1,\dots,J_r)$.
Then, the morphisms $h^r\colon \overline{M}^r\to L\dot\Omega_{X/S}^{(m),r-1}$ defined in Lemma \ref{lem:generatorbarM} satisfy
\[
h^r\big(\delta(I;J_1,\dots,J_r)\big)=(-1)^{s'-1}\delta(\hat{I};J_1,\dots,J_{s'-1})\otimes h^{r-s'+1}\big(\delta(i_n\one_n;J_{s'},\dots,J_r)\big).
\]
\end{corollary}

\begin{proof}
This is a direct consequence of the condition (ii) of Proposition \ref{thm:homotopy} and of Lemma \ref{lem:omega2}.
\end{proof}

Now, in order to prove Proposition \ref{thm:homotopy}, we have to show that
the morphism $h^r\colon\overline{M}^r\to L\dot\Omega_{X/S}^{(m),r-1}$ factors through $L\dot\Omega_{X/S}^{(m),r}$.

Let us at first prove it in the case $r=2$.

\begin{lemma}
\label{lem:homotopyr=2}
The $\sO_X$-linear morphism $h^2\colon\overline{M}^2\to L\dot\Omega_{X/S}^{(m),1}$ constructed in Lemma \ref{lem:generatorbarM}
factors through $L\dot\Omega_{X/S}^{(m),2}$.
\end{lemma}

\begin{proof}
We have to show that the section
\[
\sum_{\substack{A+B+C=J\\ B,C\neq J}}h^2\big(\delta(I; A+B, A+C)\big)
\]
is zero when $p^m<|J|\leq 2p^m$ and $0\leq j_n<p^m$.
In case $j_n=0$, this is obvious by the definition of $h^2$, therefore let us assume that $0<j_n<p^m$.
Moreover, we may assume that $p^m|i_n$ because otherwise each term is zero
by the first condition of Proposition \ref{thm:homotopy} (i).

Then, the section above is equal to
\[
h^2\big(\delta(I;\hat{J},j_n\one_n)\big)+
\sum_{\substack{a+b+c=j_n\\ c\neq j_n}}\Gamma_{a,b,c}\,h^2\big(\delta(I;(a+b)\one_n,\hat{J}+(a+c)\one_n)\big).
\]
The first term is equal to
\begin{eqnarray*}
-\delta(\hat{I};\hat{J})\otimes \eta_n^{\{i_n+j_n\}}&=&-\hat\eta^{\{\hat{I}\}}\eta_n^{\{i_n+j_n\}}(\dlog\hat{t})^{\hat{J}}
- \hat\eta^{\{\hat{I}\}}d^0\left(\eta_n^{\{i_n+j_n\}}\right)(\dlog\hat{t})^{\hat{J}}\\
&=& -\sum_{a+b+c=i_n+j_n} \Gamma_{a,b,c}\,\hat\eta^{\{\hat{I}\}}\eta_n^{\{a+b\}}(\dlog \hat{t})^{\hat{J}}(\dlog t_n)^{a+c}
\end{eqnarray*}
In this sum, only the terms for $b\geq i_n$ and $c\neq j_n$ appear in virtue of Lemma \ref{lem:b+p^mq}
and of the fact that $(\dlog\hat{t})^{\hat{J}}(\dlog t_n)^{a+c}=0$ if $a+c\geq j_n$.
In turn, the second term is
\[
\sum_{\substack{a+b+c=j_n\\ c\neq j_n}}\Gamma_{a,b,c}\hat\eta^{\{\hat I\}}\eta_n^{\{i_n+a+b\}}(\dlog\hat{t})^{\hat{J}}(\dlog t_n)^{a+c}.
\]
Therefore, by Lemma \ref{lem:b+p^mq}, the sum of these two terms is reduced to zero.
\end{proof}

Now, let us finish the proof of Proposition \ref{thm:homotopy}.

We have to prove that the morphism $h^r\colon\overline{M}^r\to L\dot\Omega_{X/S}^{(m),r-1}$ sends all relations of type $(s,k)$ to zero.

If $s>k$, this is obvious by Corollary \ref{cor:initialh}.
If $s=k$, we may assume that $0<j_n<p^m$ by the definition of $\overline{M}^r$.
In this case, the image by $h^r$ of $\delta(I; J_1,\dots,J_{k-1},A+B,A+C,J_{k+2},\dots,J_r)$ is by definition the following section:
\[
(-1)^{k-1}\delta(\hat{I};J_1,\dots,J_{k-1})\otimes h^2\big(\delta(i_n\one_n; A+B,A+C)\big)\delta(0;J_{k+2},\dots,J_r).
\]
Therefore the proof is reduced to the case $r=2$, which is done in Lemma \ref{lem:homotopyr=2}.
Even when $s<k$, this equation is directly proved by using the definition if $0<j_{s,n}<p^m$.

Therefore, we may assume that $s<k$ and $J_s=p^m\one_n$.
Moreover, Corollary \ref{cor:initialh} allows us to assume that $s=k-1$.
Indeed, when $s<k-1$, consider the section
\[
h^r\big(\delta(I;J_1,\dots,J_{k-1})\otimes d^1\big(\eta^{\{J\}}\big)\otimes \delta(0;J_{k+2},\dots,J_r)\big),
\]
where $J_s=p^m\one_n$ (this is the image by $h^r$ of the relation of type $(s,k)$).
If $J_{s+1}=p^m\one_n$, this is equal to zero because
\begin{equation}
    \label{eqn:p^mp^m}
    (\dlog t_n)^{p^m}\otimes (\dlog t_n)^{p^m} = \tqbinom{2p^m}{p^m}^{-1}d\big(\eta_n^{\{2p^m\}}\big) = 0.
\end{equation}
By Lemma \ref{lem:omega2}, if $j_{s+1,n}=0$, then this equals a relation of type $(s+1,k)$,
and if $0<j_{s+1,n}<p^m$, then by Lemma \ref{lem:omega2} this equals a relation which we have already treated
(that is, a relation of type $(s,k)$ such that the $n$-th component of $s$-th multi-index is less than $p^m$).
By repeating this argument, the proof is reduced to the case $s=k-1$.
Now, Corollary \ref{cor:initialh} again allows us to assume that $s=k-1=2$.

We need to show that
\[
h^r\Big(\delta(I;p^m\one_n)\otimes d\left( (\dlog t)^{J}\right)\otimes\delta(0;J_3,\dots,J_r)\Big)=0
\]
when $p^m<|J|\leq 2p^m$. Put $K:=J+p^m\one_n$ and notice that
\begin{align*}
& \sum_{\substack{A+B+C=K\\ B,C\neq K}}\Gamma_{A,B,C}\,d\left( (\dlog t)^{A+B}\right)\otimes(\dlog t)^{A+C}\\
- & \sum_{\substack{A+B+C=K\\ B,C\neq K}}\Gamma_{A,B,C}\,(\dlog t)^{A+B}\otimes d\left( (\dlog t)^{A+C}\right)
\end{align*}
is equal to zero because this is nothing other than $(d\circ d)\left( (\dlog t)^{\{K\}}\right)$.

Now, the section
\[
\sum_{\substack{A+B+C=K\\ B,C\neq K}}\Gamma_{A,B,C}h^r\Big(\eta^{\{I\}}d\left( (\dlog t)^{A+B}\right)\otimes\delta(0;A+C,J_3,\dots,J_r)\Big)
\]
is reduced to zero. In fact, each term is obviously zero when $|A+C|>p^m$.
If it is not the case, then $|A+B|>p^m$, and the term is zero because this is the image of a relation of type $(s,1)$ with $s\geq 1$.
Therefore, combining it with the previous notice, we see that
\[
\sum_{\substack{A+B+C=K\\ B,C\neq K}}\Gamma_{A,B,C}h^r\Big(\delta(I; A+B)\otimes d\left( (\dlog t)^{A+C}\right)\otimes\delta(0;J_3,\dots,J_r)\Big)=0.
\]
The left hand side is the sum of the two sections
\[
\sum_{\substack{A+B+C=\hat K\\ C\neq\hat K}}\Gamma_{A,B,C}h^r\left(\delta(I; A+B)\otimes d\big( (\dlog\hat{t})^{A+C}(\dlog t_n)^{j_n+p^m}\big)
\otimes\delta(0;J_3,\dots,J_r)\right)
\]
and
\[
\sum_{\substack{a+b+c=j_n+p^m\\ c\neq j_n+p^m}}\Gamma_{a,b,c}h^r\left(\delta(I; (a+b)\one_n)\otimes d\big( (\dlog\hat{t})^{\hat{K}}(\dlog t_n)^{a+c}\big)\otimes\delta(0;J_3,\dots,J_r)\right),
\]
because the other terms vanish by the definition of $h^r$.

Now, each term of the first section is zero because this is the image of a relation of type $(2,2)$; therefore the second section is itself equal to zero.
Also, all terms of the second section reduce to zero except ones for $a+b=p^m$.
In fact, if $a+b<p^m$, then since $|\hat{K}+(a+c)\one_n|>|J|$, the term vanishes by the known case.
Therefore we see that
\[
\sum_{d=j_n}^{j_n+p^m}\Gamma_{d-j_n,p^m-d+j_n,j_n}\,h^r\left(\delta(I;p^m\one_n)\otimes d\big((\dlog\hat{t})^{\hat{J}}(\dlog t_n)^d\big)\otimes\delta(0;J_3,\dots,J_r)\right)
\]
is equal to zero.
In this sum, only the terms for $d\leq 2p^m$ appear.
Therefore, if $j_n=2p^m$, this is just what we want.
General case follows by descending induction on $j_n$.

This completes the proof of Proposition \ref{thm:homotopy}.

\section{Proof of Proposition \ref{thm:formalLSQ}.}
\label{ss:logPL}

Now, let us prove that the family of morphisms $\{h^r\}_r$ defined in the previous section
indeed satisfies the condition in Proposition \ref{thm:formalLSQ},
namely, that it is a homotopy connecting $\id_{L\dot\Omega_{X/S}^{(m),\bullet}}$ and $\pi_n$.
We first prove it in some special cases.

\begin{lemma}
\label{lem:r=0}
The morphism $h^1$ satisfies the following equation:
\[
(h^1\circ d^0)\left(\eta^{\{I\}}\right)=
\begin{cases}
\eta^{\{I\}} & \mathrm{if}~i_n>0,\\
0 & \mathrm{if}~i_n=0.
\end{cases}
\]
\end{lemma}

\begin{proof}
By definition, we have
\begin{eqnarray*}
(h^1\circ d^0)\left(\eta^{\{I\}}\right) &=&
\sum_{\substack{A+B+C=I\\ B\neq I}}\Gamma_{A,B,C}\,h^1\big(\delta(A+B; A+C)\big)\\
&=& \sum_{\substack{a+b+c=i_n\\ b\neq i_n}}\Gamma_{a,b,c}\,h^1
\left(\hat\eta^{\{\hat{I}\}}\delta_n(a+b; a+c)\right)\\
&=& \sum_{\substack{a+b+c=i_n\\ b\neq i_n}}\Gamma_{a,b,c}\,\hat\eta^{\{\hat{I}\}}
h^1\big(\delta_n(a+b; a+c)\big).
\end{eqnarray*}
Therefore, it suffices to prove the proposition for $n=1$.

Since $d^0(1)=0$ is obvious, it is sufficient to prove that
\begin{equation}
\label{eqn:d1}
\sum_{\substack{a+b+c=p^mq+r\\ b\neq p^mq+r}}\Gamma_{a,b,c}\,h^1\big(\delta(a+b; a+c)\big)=\eta^{\{p^mq+r\}}
\end{equation}
for all natural number $q$ and $0<r\leq p^m$.
The proof is divided into three cases.

{\em Case 1:} $0<r<p^m$.

The summand is zero unless $a+b=p^mq$ and $c=r$ because in that case $h^1\big(\delta(a+b;a+c)\big)$ is zero.
If $c=r$, Lemma \ref{lem:1modp} shows that
\[
\Gamma_{a,b,c}=\qbinom{p^mq+r}{r}\binom{p^mq}{a}=\binom{p^mq}{a}.
\]
Consequently, the only remaining term is the term for $(a,b,c)=(0,p^mq,r)$, which equals
$h^1\big(\delta(p^mq; r)\big)=\eta^{\{p^mq+r\}}$.

{\em Case 2:} $r=p^m$ and $q+1$ is divisible by $p$.

In this case, by Lemma \ref{lem:b+p^mq}, each term vanishes unless $(a,b,c)=(0,p^mq,p^m)$ or $(p^m,p^mq,0)$.
In the latter case, $\Gamma_{a,b,c}=q+1=0$ by assumption, hence the only remaining term is
that  for $(a,b,c)=(0,p^mq,p^m)$,
which equals $h^1\big(\delta(p^mq; p^m)\big) = \eta^{\{p^mq+p^m\}}$ because $\sigma(q)=q+1$.

{\em Case 3:} $r=p^m$ and $q+1$ is not divisible by $p$.

In this case, we have two remaining terms as in Case 2, that is, the terms for $(a,b,c)=(0,p^mq,p^m)$ and $(p^m,p^mq,0)$.
Since $\sigma(q)=\sigma(q+1)$ holds,
the left hand side of (\ref{eqn:d1}) is
\begin{eqnarray*}
&&h^1\big(\delta(p^mq; p^m)\big)+(q+1)h^1\big(\delta(p^mq+p^m; p^m)\big)\\
&=& \eta^{\{p^mq+p^m\}}+\sum_{u=q+1}^{\sigma(q)-1}(-1)^{u-q}\left(\frac{\,u!\,}{q!}-\frac{u!}{(q+1)!}(q+1)\right)\eta^{\{p^m(u+1)\}}=\eta^{\{p^mq+p^m\}}.
\end{eqnarray*}
This shows the equation (\ref{eqn:d1}).
\end{proof}

\begin{lemma}
\label{lem:r=2}
Assume that $n=1$. Then, the equation
\[
(h^2\circ d^1+d^0\circ h^1)\big(\delta(i;k)\big) = \delta(i;k)
\]
holds for all $i\in\bN$ and $0<k\leq p^m$.
\end{lemma}

\begin{proof}
Let us write $i=p^mq+r$ with $0\leq r<p^m$.
The proof is divided into several cases depending on $q$, $r$ and $k$.

{\em Case 1:} $0<r<p^m$ ($q, k$ are arbitrary).

$h^1\big(\delta(i;k)\big)=0$ by the first half of (i) of Proposition \ref{thm:homotopy}.
In turn, $(h^2\circ d^1)\big(\delta(i; k)\big)$ is equal to
\[
h^2\left(d^0\big(\eta^{\{i\}}\big)\otimes(\dlog t)^k\right)+h^2\left(\eta^{\{i\}}d^1\big( (\dlog t)^k\big)\right),
\]
and the second term vanishes by the same reason.
Now, as in the calculation in Case 1 of the proof of Lemma \ref{lem:r=0},
we can show that the first term equals
\[
h^2\left(\delta(p^mq;r,k)\right) = \delta(p^mq+r;k).
\]

{\em Case 2:} $r=0$, $q=0$ and $0<k<p^m$.

In this case,
\begin{eqnarray*}
(h^2\circ d^1)\big( \delta(0;k)\big)&=&-\sum_{\substack{a+b+c=k\\ b,c\neq k}}\frac{k!}{a!\,b!\,c!}h^2\big( \delta(0;a+b,a+c)\big)\\
&=& -\sum_{\substack{a+b+c=k\\ b,c\neq k}}\frac{k!}{a!\,b!\,c!}\delta(a+b;a+c)
\end{eqnarray*}
and
\[
	(d^0\circ h^1)\big( \delta(0;k)\big)= d^0\big( \eta^{\{k\}}\big)
 = \sum_{\substack{a+b+c=k\\b\neq k}}\frac{k!}{a!\,b!\,c!}\delta(a+b;a+c).
\]
This shows the equation.

{\em Case 3:} $r=0$, $q\geq 1$ and $0<k<p^m$.

In this case,
\[
    (d^0\circ h^1)\left(\delta(p^mq;k)\right)=d^0\left(\eta^{\{p^mq+k\}}\right)=\sum_{\substack{a+b+c=p^mq+k\\b\neq p^mq+k}}\Gamma_{a,b,c}\delta(a+b;a+c),
\]
which is by Lemma \ref{lem:b+p^mq} equal to
\[
\sum_{\substack{a+b+c=k\\b\neq k}}\Gamma_{a,b,c}\,\delta(p^mq+a+b;a+c) + q\delta(p^mq;p^m) + \delta(p^m(q-1);p^m).
\]
Next, we have
\begin{eqnarray*}
(h^2\circ d^1)\left(\delta(p^mq;k)\right)
&=& \sum_{\substack{a+b+c=p^mq\\ b\neq p^mq}}\Gamma_{a,b,c}h^2\big(\delta(a+b; a+c,k)\big)\\
&& -\sum_{\substack{a+b+c=k\\ b,c\neq k}}\Gamma_{a,b,c}h^2\big(\delta(p^mq; a+b, a+c)\big).
\end{eqnarray*}
The second term is equal to
\[
-\sum_{\substack{a+b+c=k\\ b,c\neq k}}\Gamma_{a,b,c}\eta^{\{p^mq+a+b\}}(\dlog t)^{a+c}.
\]
The first term is, again by Lemma \ref{lem:b+p^mq}, equal to
\begin{eqnarray*}
&&qh^2\big(\delta(p^mq; p^m,k)\big) + h^2\big(\delta(p^m(q-1); p^m, k)\big)\\
&=& -q\delta(p^mq+k; p^m)-\delta(p^m(q-1)+k; p^m).
\end{eqnarray*}
This shows the assertion.

{\em Case 4:} $r=0$, $q\geq 1$ and $k=p^m$.

First, we show that $d^1\big(\delta(p^mq; p^m)\big)$ is equal to zero.
In fact, this is equal to
\[
d^1\big(\delta(p^mq; p^m)\big)=d^0\big(\eta^{\{p^mq\}}\big)\otimes(\dlog t)^{p^m}+\eta^{\{p^mq\}}d^1\big( (\dlog t)^{p^m}\big),
\]
and we show that these two terms are zero themselves.
The first term is
\[
\sum_{\substack{a+b+c=p^mq\\ b\neq p^mq}}\Gamma_{a,b,c}\,\delta(a+b; a+c, p^m).
\]
Here, $\Gamma_{a,b,c}$ vanishes unless $(a,b,c)=(p^m,p^mq-p^m,0)$ or $(0,p^mq-p^m,p^m)$ by Lemma \ref{lem:b+p^mq},
but in both cases, $\delta(a+b; a+c,p^m)$ is equal to zero by (\ref{eqn:p^mp^m}).
The second term is
\[
-\sum_{\substack{a+b+c=p^m\\ b,c\neq p^m}}\Gamma_{a,b,c}\,\delta(p^mq;a+b,a+c).
\]
The coefficient $\Gamma_{a,b,c}$ vanishes unless $(a,b,c)=(p^m,0,0)$, but then $\delta(p^mq;a+b,a+c)$ vanishes again by (\ref{eqn:p^mp^m}).
This shows the assertion stated above.

In turn, $(d^0\circ h^1)\big(\delta(p^mq;p^m)\big)$ equals
\begin{eqnarray*}
&& \sum_{u=q}^{\sigma(q)-1}(-1)^{u-q}\frac{\,u!\,}{q!}d^0\big(\eta^{\{p^m(u+1)\}}\big)\\
&=&\sum_{u=q}^{\sigma(q)-1}(-1)^{u-q}\frac{\,u!\,}{q!}
\Big\{\delta(p^mu; p^m)+(u+1)\delta(p^m(u+1); p^m)\Big\}\\
&=& \delta(p^mq; p^m).
\end{eqnarray*}

This completes the proof of the lemma.
\end{proof}

\begin{lemma}
\label{lem:finalspecial}
The following equation holds for an arbitrary $n$:
\[
(h^2\circ d^1+d^0\circ h^1)\left(\eta_n^{\{i\}}(\dlog t)^J\right) =
\begin{cases}
0 & \mathrm{if}~i=j_n=0,\\
\eta_n^{\{i\}}(\dlog t)^J & \mathrm{otherwise}.
\end{cases}
\]
\end{lemma}

\begin{proof}
If $\hat{J}=0$, then this is the same as Lemma \ref{lem:r=2}, thus we assume that $\hat{J}\neq 0$.
Then, since $h^1\big(\eta_n^{\{i\}}(\dlog t)^J\big)=0$, we calculate the section
\begin{equation}
    (h^2\circ d^1)\big(\eta_n^{\{i\}}(\dlog t)^J\big).
    \label{eqn:h2d1}
\end{equation}

First, we assume that $j_n\neq 0$. Then, the section (\ref{eqn:h2d1}) equals
\begin{eqnarray*}
&& -h^2\big(\delta(i\one_n; \hat{J}, j_n\one_n)\big)
- \sum_{\substack{a+b+c=j_n\\ c\neq j_n}}\Gamma_{a,b,c}h^2\left(\delta(i\one_n; (a+b)\one_n, \hat{J}+(a+c)\one_n)\right)\\
&+& \sum_{\substack{u+v+w=i\\ v\neq i}}\Gamma_{u,v,w}h^2\big(\delta( (u+v)\one_n; (u+w)\one_n, J)\big).
\end{eqnarray*}
This is equal to
\begin{eqnarray*}
	&& (\dlog\hat{t})^{\hat{J}}h^1\big(\delta_n(i;j_n)\big) 
- \sum_{\substack{a+b+c=j_n\\ b,c\neq j_n}} \Gamma_{a,b,c}h^2\big(\delta_n(i; a+b,a+c)\big)(\dlog\hat{t})^{\hat{J}}\\
&-& h^1\big(\delta_n(i;j_n)\big)(\dlog\hat{t})^{\hat{J}}
+ \sum_{\substack{u+v+w=i\\ v\neq i}}\Gamma_{u,v,w}h^2\big(\delta_n(u+v; u+w,j_n)\big)(\dlog\hat{t})^{\hat{J}}.
\end{eqnarray*}
Now, 
\[
(h^2\circ d^1)\left(\delta_n(i; j_n)\right) (\dlog\hat{t})^{\hat{J}}
\]
is equal to the sum of the second and the fourth term by a direct calculation, and
\[
(d^0\circ h^1)\left(\delta_n(i; j_n)\right) (\dlog\hat{t})^{\hat{J}}
\]
is equal to that of the first and the third term by (\ref{eqn:trans}).
Therefore, Lemma \ref{lem:r=2} shows that the sum equals
\[
    \delta_n(i;j_n)(\dlog\hat{t})^{\hat{J}}=\eta_n^{\{i\}}(\dlog t)^J.
\]

Next, let us assume that $j_n=0$. In this case, (\ref{eqn:h2d1}) equals
\begin{align*}
    & \sum_{\substack{a+b+c=i\\ b\neq i}}\Gamma_{a,b,c}h^2\left(\delta_n(a+b; a+c)\otimes (\dlog\hat{t})^{\hat{J}}\right)\\
    - & \sum_{\substack{A+B+C=J\\ B,C\neq J}}\Gamma_{A,B,C}h^2\big(\eta_n^{\{i\}}\delta(0; A+B, A+C)\big),
\end{align*}
and the second term is zero. Therefore this is equal to
\[
(h^1\circ d^0)\left(\eta_n^{\{i\}}\right)\otimes(\dlog{t})^J,
\]
which shows the proposition with the aid of Lemma \ref{lem:r=0}.
\end{proof}

Now, we are ready to prove Proposition \ref{thm:formalLSQ}.
In order to reduce the proof to Lemma \ref{lem:finalspecial}, we use the following lemma.

\begin{lemma}
\label{lem:induction}
Let $J, J_2,\dots, J_r \in\fI_n^{(m)}$, let $i\in\bN$, and assume that $0<j_n<p^m$.
Then, the following equation holds:
\[
h^{r+1}\left( d^1\big(\eta_n^{\{i\}}(\dlog t)^J\big)\otimes \Delta\right)
= (h^2\circ d^1)\big(\eta_n^{\{i\}}(\dlog t)^J \big)\otimes \Delta,
\]
where $\Delta:=\delta(0;J_2,\dots,J_r)$.
\end{lemma}

\begin{proof}
We use the equation
\[
d^1\big(\eta_n^{\{i\}}(\dlog t)^J\big)=d^0\big(\eta_n^{\{i\}}\big)\otimes(\dlog t)^J+\eta_n^{\{i\}}d^1\left( (\dlog t)^J\right).
\]

First, we prove that
\begin{equation}
    \label{eqn:comp1}
    h^{r+1}\left(d^0\big(\eta_n^{\{i\}}\big)\otimes(\dlog t)^J\otimes\Delta\right)
    = h^2\left(d^0\big(\eta_n^{\{i\}}\big)\otimes(\dlog t)^J\right)\otimes\Delta.
\end{equation}
We know that
\[
d^0\left(\eta_n^{\{i\}}\right)\otimes(\dlog t)^J=\sum_{\substack{a+b+c=i\\ c\neq i}}\Gamma_{a,b,c}\eta_n^{\{a+b\}}(\dlog t_n)^{a+c}
\otimes (\dlog t)^J,
\]
and the similar equation as (\ref{eqn:comp1}) holds for each term
(if $a+c=p^m$, then we have to use Lemma \ref{lem:omega2}).
Therefore (\ref{eqn:comp1}) is true.

Next, we prove that
\begin{equation}
    \label{eqn:comp2}
    h^{r+1}\left(\eta_n^{\{i\}}d^1\left( (\dlog t)^J\right)\otimes\Delta\right)
    = h^2\left(\eta_n^{\{i\}}d^1\left((\dlog t)^J \right)\right)\otimes\Delta.
\end{equation}
Here, the left-hand side is equal to
\[
-h^{r+1}\left(\delta(i\one_n; \hat{J}, j_n\one_n)\otimes\Delta\right)
-\sum_{\substack{a+b+c=j_n\\ c\neq j_n}}h^{r+1}\left(\delta(i\one_n; (a+b)\one_n, \hat{J}+(a+c)\one_n)\otimes\Delta\right).
\]
Then similarly this shows the equation (\ref{eqn:comp2}).

The two equations (\ref{eqn:comp1}) and (\ref{eqn:comp2}) show the lemma.
\end{proof}

Let us finish the proof of Proposition \ref{thm:formalLSQ}.
Put $\Delta=\hat\delta(\hat{I}; J_1,\dots,J_{k-1}), H:=\delta(i_n\one_n; J)$ and $\Delta':=\delta(0; J_{k+1},\dots,J_r)$, where
$I\in\bN^n$, $J_1,\dots,J_{k-1}\in\fI_{n-1}^{(m)}$ and $J, J_{k+1},\dots,J_r\in\fI_n^{(m)}$.
Assume that $k=r$ or that $0<j_n<p^m$.

Then, such sections $\Delta\otimes H\otimes \Delta'$ generate $L\dot\Omega_{X/S}^r$ by Lemma \ref{lem:generatorbarM} and (\ref{eqn:trans}),
therefore let us compute the image of these through $h^{r+1}\circ d^r+d^{r-1}\circ h^r$.

Now, $(h^{r+1}\circ d^r)\left(\Delta\otimes H\otimes\Delta'\right)$ is equal to the section
\begin{align*}
    &h^{r+1}\big(d^{k-1}(\Delta)\otimes H\otimes\Delta'\big)
    + (-1)^{k-1}h^{r+1}\big(\Delta\otimes d^1(H)\otimes\Delta'\big)\\
    & \hspace{3pt} + (-1)^kh^{r+1}\big(\Delta\otimes H\otimes d^{r-k}(\Delta')\big)\\
    =& (-1)^k d^{k-1}(\Delta)\otimes h^1(H)\Delta' + \Delta\otimes h^{r-k+1}\left(d^1(H)\otimes \Delta'\right)
    - \Delta\otimes h^1(H) d^{r-k}(\Delta').
\end{align*}
By Lemma \ref{lem:induction}, the second term is equal to $\Delta\otimes (h^2\circ d^1)(H)\otimes \Delta'$.
In turn,
\[
(d^{r-1}\circ h^r)\left(\Delta\otimes H\otimes\Delta'\right)=(-1)^{k-1} d^{k-1}\left(\Delta\otimes h^1(H)\Delta'\right).
\]
Then, this is
\[
(-1)^{k-1}d^{k-1}(\Delta)\otimes h^1(H)\Delta'+\Delta\otimes (d^0\circ h^1)(H)\otimes\Delta'
+ \Delta\otimes h^1(H) d^{r-k}(\Delta').
\]
By adding two sections, we see that
\[
(h^{r+1}\circ d^r+d^{r-1}\circ h^r)\left(\Delta\otimes H\otimes \Delta'\right)=\Delta\otimes(h^2\circ d^1+d^0\circ h^1)(H)\otimes\Delta'.
\]
This equation, Lemma \ref{lem:finalspecial} and Lemma \ref{lem:induction} deduce the proposition.

\section{Application to Cohomology.}

Let $(S,\fa,\fb,\gamma)$ be a fine log.\ $m$-PD scheme and $X$ a fine log.\ scheme over $S$ such that
the underlying scheme $\underline{X}$ is flat over $\underline{S}$.

Theorem \ref{thm:trueformalLSQ} gives the log.\ exact Poincar\'e lemma.

\begin{theorem}
    Assume that $X$ is log.\ smooth over $S$ and that $p\sO_X=0$.
    Let $M$ be a flat $\sO_{X/S}^{(m)}$-module.
    Then, the complex $M\otimes \bar{L}^{(m)}(\dot\Omega_{X/S}^{(m),\bullet})$ is a resolution of $M$.
    \label{thm:cohomology}
\end{theorem}

\begin{proof}
    Because of Proposition \ref{thm:linearization} (i) and Lemma \ref{lem:inducingstrat},
    the complex $\bar{L}^{(m)}(\dot\Omega_{X/S}^{(m),\bullet})$ of $\sO_{X/S}^{(m)}$-modules
    gives $\bar{L}^{(m)}\dot\Omega_{X/S}^{(m),\bullet}$.
    Combining with the natural $\sO_{X/S}^{(m)}$-linear map $\sO_{X/S}^{(m)}\to \bar{L}^{(m)}(\dot\Omega_{X/S}^{(m),0})$,
    we get a morphism $M\to M\otimes\bar{L}^{(m)}(\dot\Omega_{X/S}^{(m),\bullet})$.

    In order to prove that this is an isomorphism, it suffices to argue on each fundamental
    log.\ $m$-PD thickening $(U,T,J,\delta)$, and the assertion is local on $T_{\et}$.
    We may therefore, by Lemma \ref{lem:rest-D}, assume that there exists an $S$-morphism $h\colon T\to X$ compatible with $i\colon U\hookrightarrow T$ and that the underlying morphism of $h$ is flat.
    On such a $T$, this morphism is of the form
    $M_T\to M_T\otimes_{\sO_T}h^{\ast}\big(L\dot\Omega_{X/S}^{(m),\bullet}\big)$,
    and it is obtained from the natural morphism $\sO_X\to L\dot\Omega_{X/S}^{(m),\bullet}$
    by pull-back along $h$ and by taking tensor with $M_T$.
    The morphism $\sO_X\to L\dot\Omega_{X/S}^{(m),\bullet}$ is an quasi-isomorphism by Theorem \ref{thm:trueformalLSQ}.
    The flatness of the underlying morphism of $h$ and that of $M_T$ shows the assertion.
\end{proof}

\begin{corollary}
\label{thm:exactpoincare}
    Assume that $X$ is log.\ smooth over $S$ and that $p\sO_X=0$.
    Let $E$ be a flat log.\ $m$-crystal in $\sO_{X/S}^{(m)}$-modules.
    Then, there exists an isomorphism in the derived category
    \[
    \bR{u_{X/S}^{(m)}}_*(E)\to E_X\otimes\dot\Omega_{X/S}^{(m),\bullet}.
    \]
\end{corollary}

\begin{proof}
    This follows from Theorem \ref{thm:cohomology} by a standard argument.
\end{proof}

\begin{theorem}
    Let $(\fa_0,\fb_0,\gamma_0)$ be a quasi-coherent $m$-PD sub-ideal of $\fa$, and $S_0$ the exact closed subscheme of $S$
    defined by $\fa_0$. We assume that there exists a fine log.\ smooth scheme $Y$ over $S$ such that
    the underlying scheme $\underline{Y}$ is proper and flat over $\underline{S}$ and that $X=Y\times_SS_0$.
    Moreover, we assume that $p\sO_Y=0$ and that $\underline{S}$ is noetherian.
    Let $f\colon X\to S$ denote the structure morphism. Then, for each natural number $i$ and each flat $\sO_{X/S}^{(m)}$-module $E$ of finite presentation,
    the $\sO_S$-module $R^i{f_{X/S}^{(m)}}_*E$ is finitely generated.
    \label{thm:finitude}
\end{theorem}

\begin{proof}
    Since the question is local, the quasi-coherence of $E$ and the cohomology 
    long exact sequence allows us to assume that $E$ is free.
    Moreover, by using Proposition \ref{thm:lift}, we may assume that $S=S_0$.
    then, the assertion follows from Corollary \ref{thm:exactpoincare} as in the classical case.
\end{proof}